\documentclass[12pt]{amsart}
\usepackage{amsmath, amssymb, amsthm, amsfonts, mathrsfs}
\usepackage{cite}
\usepackage{amscd}
\usepackage{url}
\usepackage{graphicx}
\usepackage{caption}
\usepackage{subcaption}

\newtheorem{defn}{Definition}
\newtheorem{lemma}{Lemma}

\newtheorem{theorem}{Theorem}

\theoremstyle{definition}

\newcommand{\R}{\mathbb{R}}
\newcommand{\ess}[1]{\hat{#1}}
\newcommand{\escort}{\ensuremath{\phi}}
\newcommand{\escortdist}{\ensuremath{\hat{\escort}}}
\newcommand{\escortlog}{\ensuremath{\log_{\escort}}}
\newcommand{\escortlogi}{\ensuremath{\log_{\escort_i}}}
\newcommand{\escortdiv}{D^{\escort}}
\newcommand{\ecross}[2]{{#1}_i \escortlogi{#2}_i}
\newcommand{\escortexp}{\ensuremath{\exp_{\escort}}}
\newcommand{\escortexpi}{\ensuremath{\exp_{\escort_i}}}
\newcommand{\incentive}{\varphi}
\newcommand{\tsd}[1]{{#1}^{\Delta}}

\begin{document}
% \title{Time Scale Escort Incentive Dynamics}
\title{Stability of Evolutionary Dynamics on Time Scales}
\author{Marc Harper}
% \address{University of California Los Angeles}
% \email{marcharper@ucla.edu} 
\author{Dashiell Fryer}
% \authorinfo{\texttt{c}}
\date{\today}
% \subjclass[2000]{Primary: 37N25; Secondary: 91A22, 94A15}
% \keywords{evolutionary game theory, information geometry, information divergence, replicator equation, Bayesian inference, information geometry, Fisher information}

\begin{abstract}
We combine incentive, adaptive, and time-scale dynamics to study multipopulation dynamics on the simplex equipped with a large class of Riemmanian metrics, simultaneously generalizing and extending many dynamics commonly studied in dynamic game theory and evolutionary dynamics. Each population has its own geometry, method of adaptation (incentive), and time-scale (discrete, continuous, and others). Using an information-theoretic measure of distance we give a widely-applicable Lyapunov result for the dynamic. We include a wealth of examples leading up to and beyond the main results.
\end{abstract}

\maketitle

\section{Introduction}

Evolutionary dynamics now includes the study of many discrete and continuous dynamical systems such as the replicator\cite{Taylor78}, best reply  \cite{Gilboa91}, projection \cite{Nagurney96} \cite{Sandholm08} \cite{Joosten2008}, and logit dynamics \cite{Fudenberg98}, to name a few. Modified population growth dynamics incorporating a scale-invariant power law parameter, commonly used in generalized statistical physics \cite{tsallis1988possible}, have recently been applied to human populations \cite{Hernando12} and are closely related to the dynamical systems described in \cite{Harper2011}. We explore the unification of all these dynamics with the game-theoretically motivated incentive dynamic introduced in \cite{Fryer2012}, in discrete and continuous forms using time-scale calculus \cite{Bohner-Peterson-2003}, with Riemannian geometries on the simplex, specified by a geometrically motivated functional parameter called an escort \cite{Harper2011} that allows some dynamics to be realized in multiple ways and by a more arbitrary Riemannian metric in general. Moreover, we show that the incentive dynamic and the replicator dynamic are equivalent, through a mapping that yields insight into the stability of the aforementioned dynamics, and explains clearly how to separate the selection action from the underlying geometry, allowing for e.g. best-reply projection dynamics. Ultimately we define the time-scale escort incentive dynamic and time-scale metric dynamics (which in the continuous case correspond to a special case of the adaptive dynamics of \cite{Hofbauer90}), building up through a series of examples, and prove a general stability theorem for a large class of discrete and continuous dynamics. For general introductions to evolutionary dynamics see \cite{Hofbauer03} \cite{Cressman03} \cite{Hofbauer98}.

In this paper we show that the Kullback-Leibler information divergence and natural generalizations serve as Lyapunov functions in a variety of contexts. The use of the KL-divergence and similar information-theoretic quantities in evolutionary dynamics goes back at least to \cite{Bomze91} and is developed further in \cite{Weibull98}. A geometrically-motivated generalization from information geometry \cite{Amari93} was introduced in \cite{Harper2011}. To the reader familiar with information geometry, this should come as little surprise since the Shahshahani metric of evolutionary game theory can be identified with the Fisher information metric, which is in some sense a local variant of the KL-divergence. The projection dynamic can similarly be described in terms of the Euclidean geometry \cite{Lahkar08} and stability in terms of the Euclidean distance \cite{Nagurney97} which can be realized as a generalized information geometry \cite{Harper2011} through the introduction of a functional parameter called an escort. We show in this work that the geometry can simplify the stability theory of some evolutionary dynamics in addition to defining new dynamics, and is compatible with the formulation of the incentive dynamic and the time scale calculus. Finally, we introduce an information divergence for a general Riemannian metric satisfying some mild assumptions, extending the adaptive dynamics of \cite{Hofbauer90} to discrete time scales. Together, these ingredients yield a very general stability result.

\section{Incentive Dynamics}

Let us first introduce the incentive dynamics of Fryer \cite{Fryer2012}. Motivated by game-theoretic considerations, the incentive dynamics takes the form
\begin{equation}
\dot{x_i} = \incentive_i(x) - x_i \sum_{j}{\incentive_j(x)}.
\label{incentive_dynamic}
\end{equation}
Table \ref{incentives_table} lists incentive functions for many common dynamics. Fryer shows in \cite{Fryer2012} that any game dynamic can be written as an incentive dynamic and gives a stability theorem as follows. Define an \emph{incentive stable state} (ISS) to be a state $\ess{x}$ such that in a neighborhood of $\ess{x}$ the following inequality holds
\begin{equation}
 \sum_{i}{ \frac{\ess{x}_i \incentive_i(x)}{x_i}} > \sum_{i}{ \frac{x_i \incentive_i(x_i)}{x_i}}.
 \label{iss} 
\end{equation}
It is then shown that given an internal ISS, the Kullback-Liebler information divergence is a local Lyapunov function for the incentive dynamic (Equation \ref{incentive_dynamic}. The incentive $\incentive_i(x) = x_i f_i(x)$ captures the known result for the replicator dynamics \cite{Bomze91} \cite{Hofbauer98}, with the definition of ISS being exactly ESS: $\ess{x} \cdot f(x) > x \cdot f(x)$. Moreover, a short proof shows that for the best reply incentive, ISS is again ESS \cite{fryer2012kullback}.
\begin{figure}[h]
    \centering
    \begin{tabular}{|r|c|}
        \hline
        Dynamics Name & Incentive \\ \hline
        Replicator & $\incentive_i(x) = x_i \left(f_i(x) - \bar{f}(x)\right)$\\ \hline
        Best Reply & $\incentive_i(x) = BR_i(x) - x_i$ \\ \hline
        Logit & $\displaystyle{ \incentive_i(x) = \frac{\text{exp}(\eta^{-1}f_i(x))}{\sum_j{\text{exp}(\eta^{-1}f_j(x))}}}$ \\ \hline
        Projection & $\incentive_i(x) = \begin{cases}
                    & \displaystyle{f_i(x) - \frac{1}{|S(f(x), x)|}\sum_{j \in S(f(x), x)}{f_j(x)}} \quad \text{if $i \in S(f(x), x)$}\\
                    &0 \qquad \text{else}
                    \end{cases}$\\
        \hline
        \end{tabular}
    \caption{Incentives for some common dynamics, where $BR_i(x)$ is the best reply to state $x$, $S(f(x), x)$ is the set of all strategies in the support of $x$ as well as any collection of pure strategies in the complement of the support that maximize the average. Note that on the interior of the simplex the projection incentive is just $\incentive_i(x) = f_i(x) - 1/n \sum_j{f_j(x)}$. For more examples see Table 1 in \cite{Fryer2012}.}
    \label{incentives_table}
\end{figure}

\subsection{The Incentive Dynamic is the Replicator Dynamic}
Observe that we can transform any incentive dynamic into a replicator dynamic on the interior of the simplex (the behavior of incentives near the boundary of the simplex is another matter altogether and we will not consider it here). We simply solve for $g_i$ in the equation $\incentive_i(x) = x_i f_i(x)$ so that to every incentive $\incentive_i$ we define an effective fitness landscape $\frac{\incentive_i(x)}{x_i} = f_i(x)$, which is well-defined at least on interior trajectories and possibly on the boundary simplices, such as is the case for forward-invariant dynamics. (Note that some incentives take particular care on the boundary, such as the projection incentive in \cite{Sandholm08}.) The summation term $\sum_{j}{\incentive_j(x)}$ in equation (\ref{incentive_dynamic}) is the mean of the effective fitness landscape.

The ISS condition for an incentive is the same as the ESS condition for the effective fitness landscape, and this shows that the ISS stability theorem is equivalent to the analogous theorem for ESS and the replicator dynamic. This does not of course imply ESS for any fitness landscape used in the definition of an incentive (such as a best reply incentive using a landscape $f$). Typical fitness landscapes in evolutionary dynamics are linear and given by $f(x) = Ax$ where $A$ is a game matrix. In this formulation, one will encounter a much larger class of landscapes. For the best reply dynamic, the effective landscape is $f_i(x) = BR_i(x) / x_i - 1$, and while it is clear that such a function may not be well-defined on the boundary simplex, since the dynamic is forward-invariant we will not concern ourselves at this time.

\subsection{Example: q-Replicator Incentive}
From the preceeding section it is tempting to suspect that the concepts of ISS and ESS are equivalent, but this is not the case. Consider the \emph{q-replicator incentive} $\incentive_i(x) = x_i^q f_i(x)$ for a fitness landscape $f$. Further, assume that the fitness function is of the form $f(x) = Ax$ where $A$ is the rock-scissors-paper matrix:
\[ f(x) = \left( \begin{matrix}
          0 & -b & a\\
          a & 0 & -b \\
          -b & a & 0
          \end{matrix} \right) x \]
Several curves for various values of $q$ are plotted in Figure \ref{fig_1}. For the replicator incentive ($q=1$) the trajectory converges to an interior ESS; for other values of $q$, the trajectories may either converge or diverge. This shows that an ESS for the fitness landscape need not be an ISS for the incentive. Note also that whether a curve converges or not depends on the initial point. While for $q=1$ the Lyapunov function is global, this is not the case for other values of $q$. See Figure \ref{fig_2}. Figure \ref{fig_3} shows that an ISS need not be an ESS, and Figure \ref{fig_4} shows that an ESS need not be an ISS.

\begin{figure}[h]
        \begin{subfigure}[b]{0.5\textwidth}
                \centering
                \includegraphics[width=\textwidth]{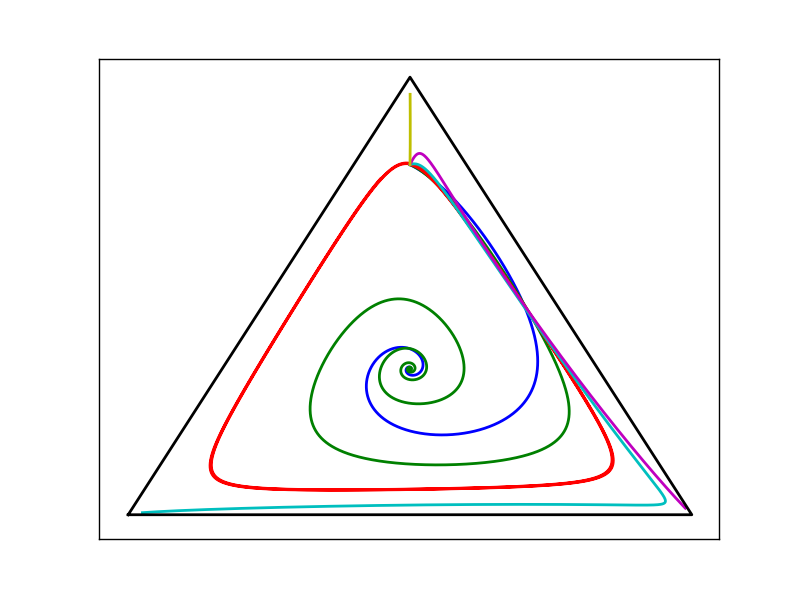}
                \caption{Initial point: $(1/10, 1/10, 8/10)$}
                \label{fig_1_1}
        \end{subfigure}%
        ~ %add desired spacing between images, e. g. ~, \quad, \qquad etc. 
          %(or a blank line to force the subfigure onto a new line)
        \begin{subfigure}[b]{0.5\textwidth}
                \centering
                \includegraphics[width=\textwidth]{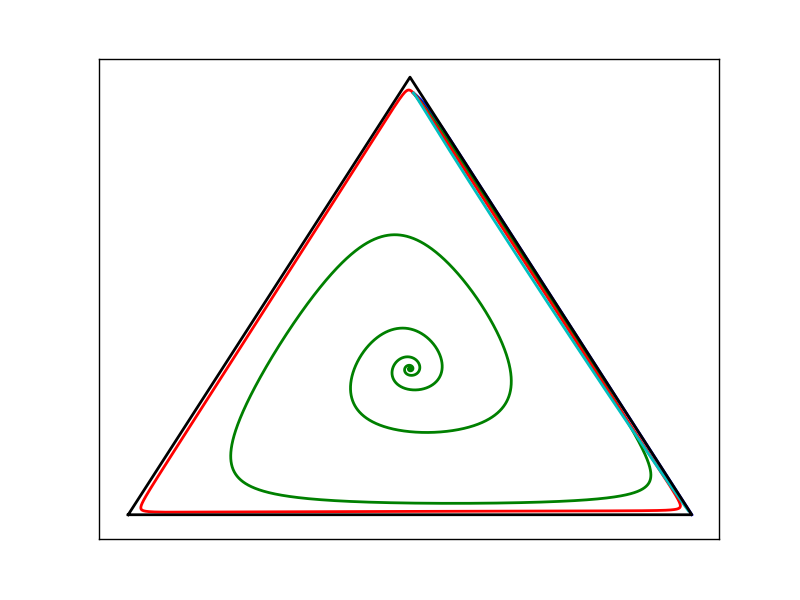}
                \caption{Initial point: $(1/83, 2/83, 80/83)$}
                \label{fig_1_2}
        \end{subfigure}
        \caption{Phase portraits for q-Replicator incentive for the powers $q \in \{0.5, 1, 1.5, 2, 2.5, 4 \}$ with colors blue, green, red, cyan, magenta, and yellow, respectively. The game matrix is the the RSP matrix with a=-1, b=-2. An ESS need not be an ISS since some of these trajectories converge and some do not, depending on the incentive.}
        \label{fig_1}
\end{figure}          

\begin{figure}[h]
        \begin{subfigure}[b]{0.5\textwidth}
                \centering
                \includegraphics[width=\textwidth]{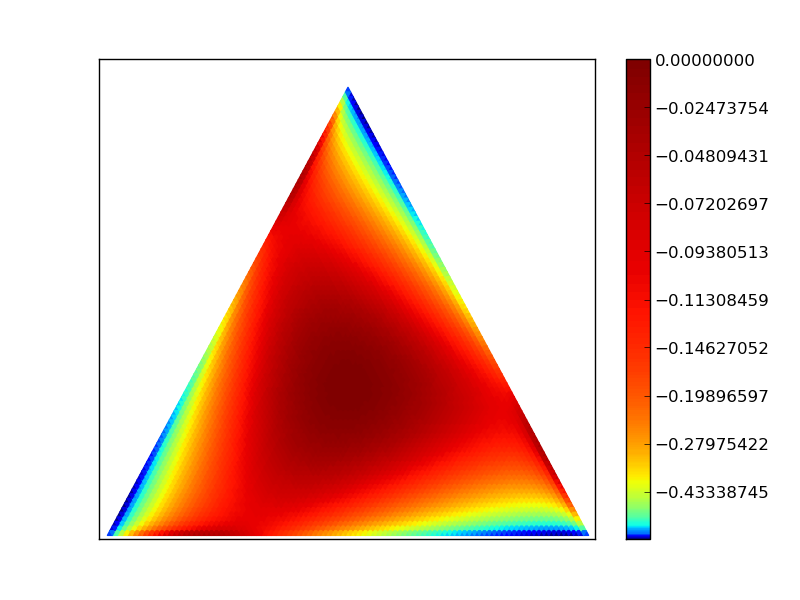}
                \caption{$q=0.78$}
                \label{fig_2_1}
        \end{subfigure}%
        ~ %add desired spacing between images, e. g. ~, \quad, \qquad etc. 
          %(or a blank line to force the subfigure onto a new line)
        \begin{subfigure}[b]{0.5\textwidth}
                \centering
                \includegraphics[width=\textwidth]{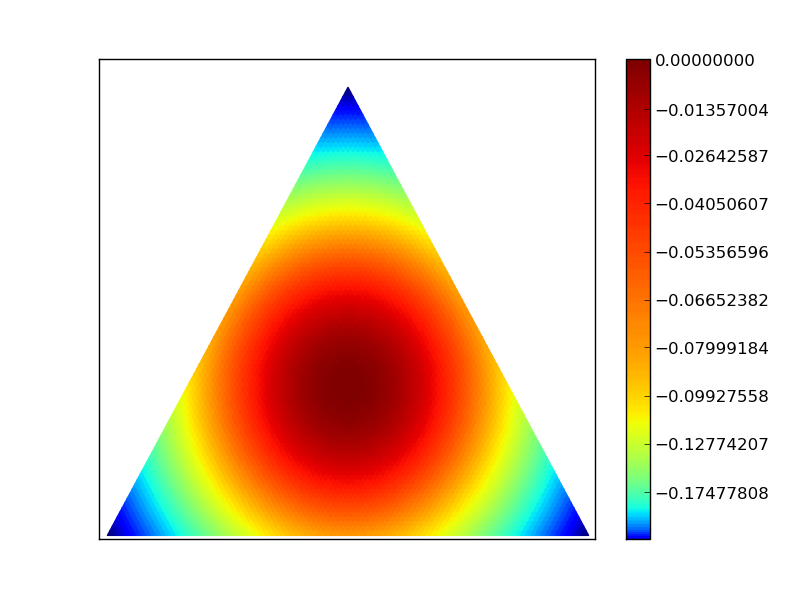}
                \caption{$q=1$}
                \label{fig_2_2}
        \end{subfigure}
        \begin{subfigure}[b]{0.5\textwidth}
                \centering
                \includegraphics[width=\textwidth]{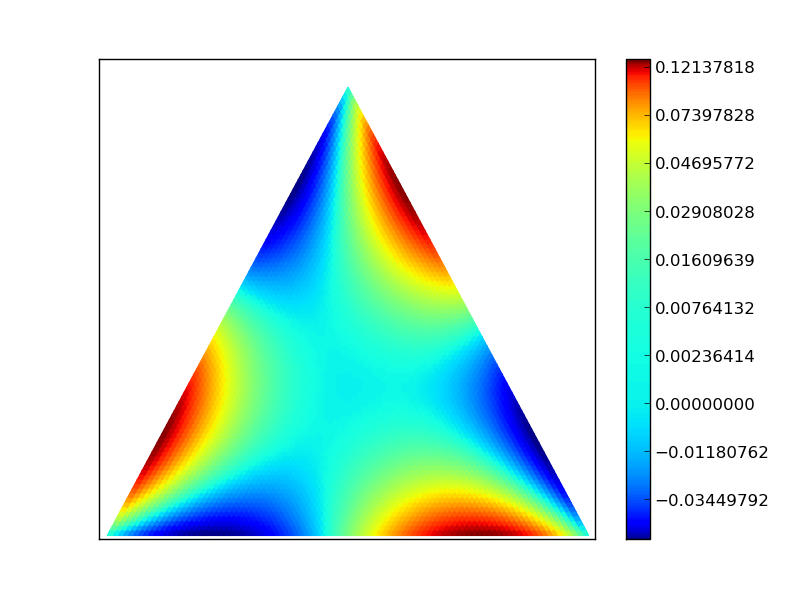}
                \caption{$q=2.5$}
                \label{fig_2_3}
        \end{subfigure}
        \caption{ISS: LHS - RHS. Stability need not be global for an ISS. Note that for $q=2.5$, the quantity changes sign depending on point. This is true for all $q > 1$.}
        \label{fig_2}
\end{figure}     

\begin{figure}[h]
    \includegraphics[width=0.5\textwidth]{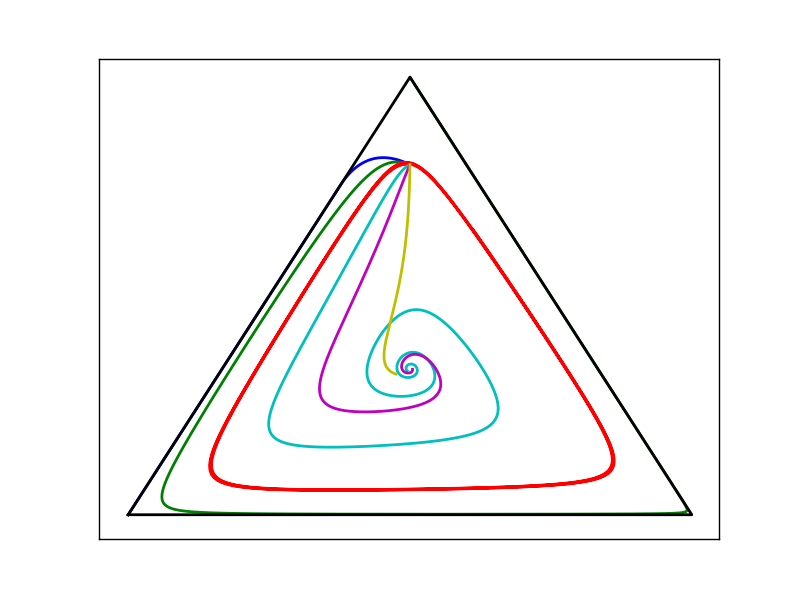}
    \caption{Phase portraits for q-Replicator incentive for the powers $q \in \{0.5, 1, 1.5, 2, 2.5, 4 \}$ with colors blue, green, red, cyan, magenta, and yellow, respectively. The game matrix is the the RSP matrix with a=1, b=2. Initial point: $(1/8, 1/8, 6/8)$. The standard replicator dynamic diverges for this landscape, so an ISS need not be an ESS.}
    \label{fig_3}
\end{figure}

\section{Time Scale Incentive Dynamics}

\subsection{Time-Scale Definitions}
To study dynamics at different time scales, a key ingredient is the time scale derivative, also known as the \emph{delta derivative}. General references for time scale dynamics include \cite{Bohner-Peterson-2001}, \cite{Bohner-Peterson-2003}. For time scales $\mathbb{T} = h \mathbb{Z}$, the time scale derivative is given by:
\[ f^{\Delta}= \frac{f(x+h) - f(x)}{h},\] which is better known as the difference quotient. For the time scale $\mathbb{T} = \mathbb{R}$, the time scale derivative is the standard derivative. While there are other common time scales, this paper restricts attention to $\mathbb{R}$ and $h\mathbb{Z}$ for $0 < h \leq 1$ for the most part.

% and the following time scale. Let $\mathbb{T} = \{t_0, t_1, \ldots\}$ such that $h_{i} = t_{i} - t_{i-1}$ is decreasing; we will call this the ficticious play time scale, after the special case in which $h_{i} = 1 / (i + 1)$. In this case, the time scale derivative is still given by the difference quotient but the denominator varies with time.

\subsection{Time Scale Replicator Equation}

Before we proceed, let us consider an illustrative example. Define the time scale replicator equation as:
\[ x_{i}^{\Delta} = \frac{x_i \left(f_i(x) - \bar{f}(x) \right)}{\bar{f}(x)}. \]
Note that $\sum_{i}{x_{i}^{\Delta}} = 0$. For $h=1$, this is the discrete replicator dynamic, where
\[ x_{i}^{\Delta} = \frac{x(t+h) - x(t)}{h} := x'_i - x_i, \]
which can be easily seen to be equivalent to the usual discrete dynamic
\[ x'_i = \frac{x_i f_i(x)}{\bar{f}(x)}.\]
For the case $h \to 0$, we have the following:
\[ \dot{x}_{i} = \frac{x_i \left(f_i(x) - \bar{f}(x) \right)}{\bar{f}(x)}. \]
This equation is trajectory equivalent to the replicator equation (up to a change in velocity and transformation of $f(x)$ so that $\bar{f}(x) > 0$ \cite{Hofbauer98} \cite{Cressman03}). The definitions of the time-scale calculus can unify the description of continuous and discrete dynamics in evolutionary dynamics, as well as setting the stage for dynamics on other time scales, through the delta derivative.

\subsection{Time Scale Best Reply and Fictitious Play Dynamic}

In \cite{Fryer2012}, a best reply dynamic is shown to result from the incentive $\incentive_i(x) = BR_i(x) - x_i$, and the ISS condition is shown to reduce to the ESS condition. Now consider the case of $\mathbb{T} = h\mathbb{Z}$ for $h \in (0,1)$. This time scale yields a discrete best reply that is algebraically equivalent to
\[ x'_i = (1-h)x_i + h BR_i(x).\]
Here we see that the time-scale appears as a weighting between the best reply and the current state of the dynamic. In other words, $h$ is the proportion of the population that adopts the best reply. If $h=1$, then entire population switches to the best reply, and the dynamic cycles through the corner points of the simplex for the landscape defined by an RSP matrix as above.

The center of the simplex is not an equilibrium point for the time-scale $\mathbb{T} = h \mathbb{Z}$. To see this, suppose the fitness landscape $f$ has an interior ESS at the barycenter (say for an RSP matrix). On the interior of the simplex, the dynamic is stationary if and only if $x_{k+1} = x_{k}$ for some $k$, which implies that either $h=0$ or $x_i \in \{0, 1\}$ for all $i$, which is impossible for an interior ESS. However, if the time-step $h$ is not fixed, this may not be the case (see \ref{fig_5}). On a variable time scale the dynamic is (in vector form)
\[ x_{k+1} = (1-h_k)x_{k} + h_k BR(x_{k}),\]
which is similar to the fictitious play dynamic in the case that the time-scale weights decrease over time. Eventually such a time scale has $h_k  \to 0$, so the above discussion does not apply, and this dynamic can converge.

\begin{figure}[h]
    \includegraphics[width=0.5\textwidth]{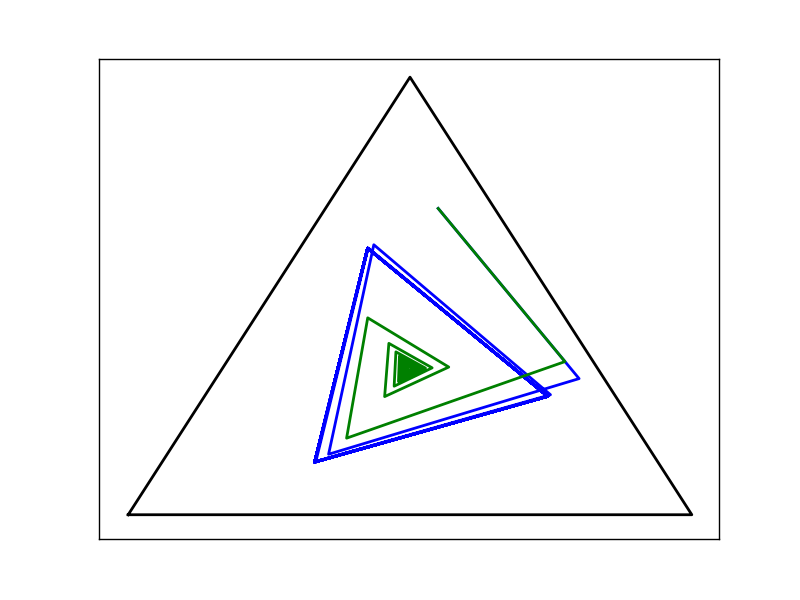}
    \caption{Best reply for $h=1/3$ (blue) and a variable time scale with $h_k = 1 / (k+1)$ (green). Only the latter converges to the interior ESS.}
    \label{fig_2_a}
\end{figure}

\subsection{Time Scale Lyapunov Function}

Lyapunov stability is a frequently-used technique in evolutionary dynamics. These techniques have been ported to the time scale calculus \cite{Bartosiewicz11} \cite{DaCunha-2005} \cite{Davis-Gravagne-Marks-Ramos-2010}. To the reader familiar with the traditional theory, the methods will be very familiar. A time scale Lyapunov function is a positive definite function defined on the trajectories of a dynamic with negative-semi-definite (negative-definite) time scale derivative which then implies stability (asymptotic stability). The interested reader should see \cite{Bartosiewicz11} for a presentation that mirrors the traditional approach with the appropriate changes necessary to extend the stability theory to arbitrary time scales. In particular, we will use Theorems 5.1 and 5.4 in \cite{Bartosiewicz11}, which have two technical conditions that will be familiar to the reader aware of the details of the continuous-time Lyapunov stability theorems. Here we give the necessary time-scale definitions and discuss the technical conditions; for more detail, see \cite{Bartosiewicz11}.
Define $K_h(\ess{x}) = \{ x \in \R^{n} \, : ||x - \ess{x}|| < h \}$. A function $c: [0, h] \to [0, \infty)$ belongs to class K if it is continuous and increasing, and if $c(0) = 0$. A function $v: K_h \times \mathbb{T} \to \R$ is called decrescent if there exists a function $c$ of class K and $t_0 \in \mathbb{T}$ such that for all $t \geq t_0$ and $x \in K_h$, $v(x, t) \leq c(|x - \ess{x}|)$. Under the same conditions, if there exists $c$ such that $v(x, t) \geq c(|x - \ess{x}|)$, $v$ is called positive definite. We will take the definition of \emph{time scale Lyapunov function} to include both conditions on the appropriate time scale; all such Lyapunov functions in this paper satisfy both conditions.

We now define a time scale generalization of the incentive dynamics and generalize the Lyapunov results of \cite{Fryer2012} and \cite{Harper2011} to cover more general time scale dynamics.

\begin{theorem}
For $0 \leq h \leq 1$, $D(x) = D(\ess{x}, x)$ is a ``time-scale Lyapunov function'' for the time scale incentive dynamic iff $\ess{x}$ is an ISS.
\label{thm_iss_lyapunov}
\end{theorem}
This theorem is a special case of Theorem \ref{thm_tsei} and so the proof is deferred. A direct proof for the replicator equation is an straightforward exercise.

\section{ISS and ESS}

A natural question is: for which incentives, $\incentive(x)$, are the stability
concepts ISS and ESS equivalent? Since an interior ISS is unique and asymptotically stable for
the incentive dynamics and an ESS is unique and asymptotically stable for many
known special cases (e.g. the replicator dynamic), either the concepts must
coincide in these cases or only one of the two criterion can hold. 

Without loss of generality, we assume $\sum_i \incentive_i(x) = 0$ for every
$x\in\Delta$. The ESS condition $x\cdot f(x) < \ess{x}\cdot f(x)$ for all $x$ in
a neighborhood of $\ess{x}$ can be rewritten as $\ess{x}\cdot[f(x) - (x\cdot f(x))\textbf{1}] >0$, where $\textbf{1}$ is a vector of all ones. Thus at
a pure ESS, the \emph{payoff positivity} requirement, $\incentive_i(x)/x_i > 0
\Leftrightarrow f_i(x) > x\cdot f(x)$, preserves not only the asymptotic
stability but also the exact basin of attraction as well. Furthermore, the set $U = \{f_i(x) > x\cdot
f(x)\} \bigcap_{j\neq i} \{f_j(x) < x\cdot f(x)\}$ is open and non-empty for a
continuous fitness near a pure ESS. Thus \emph{weak payoff positivity}
\cite{Weibull98}, where at least one above average fitness must have a
positive incentive growth rate, preserves the asymptotic stability of a pure
ESS, albeit for a smaller basin of attraction.

Samuelson and Zhang \cite{samuelson1992evolutionary} introduced a class of games
called \emph{aggregate monotone}, which trivially preserve ESS and the basin of
attraction. However, they prove in the single population case all aggregate
monotone dynamics are dynamically equivalent to the replicator equation. In the
multi-population case the dynamics for each population is a positive function
multiplied by the replicator equation. In either case it does not lead to
interesting further discussion.

\begin{defn}[Aggregate Monotone]
An incentive is \emph{aggregate monotone} if $\forall y,z\in\Delta$, \[y\cdot
f(x) > z\cdot f(x) \Leftrightarrow y\cdot \frac{\incentive(x)}{x} > z \cdot
\frac{\incentive(x)}{x}
\]
\end{defn} 

Attention may be restricted to single population games with mixed ESS, as Selten
\cite{selten1980note} established that an ESS in an asymmetric
game\footnote{The asymmetric condition includes all multi-population
formulations.} must be pure. This result assumes linear fitness which is more
restrictive than necessary for the arguments above. However, the following
results do all assume linear fitness.

First introduced by Nachbar \cite{nachbar1990evolutionary}, the concept of
payoff monotone dynamics is a reasonable starting point for investigation.
\begin{defn}[Payoff Monotone]
An incentive is \emph{payoff monotone} if $\forall i,j$,
\[f_i(x) > f_j(x) \Leftrightarrow \frac{\incentive_i(x)}{x_i} >
\frac{\incentive_j(x)}{x_j}\]
\end{defn}
Unfortunately, Friedman \cite{friedman1991evolutionary}, via a verbal
description, shows that payoff monotone dynamics need not be asymptotically stable at an
internal ESS.

Fortunately, Cressman \cite{cressman1997local} shows that for smooth incentives
the condition of payoff positivity will preserve the asymptotic stability of an
ESS as long as the linearization is not trivial. His results are paraphrased in
this theorem for completeness.

\begin{theorem}
An interior ESS, $\ess{x}$, is asymptotically stable for any smooth payoff
positive dynamic which has a nontrivial linearization about $\ess{x}$.
\label{thm_ess_iss}
\end{theorem}

\section{Escorts}

The stability theorem for the incentive dynamic reduces the problem of finding a Lyapunov function to the problem of proving the ISS condition valid for a candidate equilibrium. Consider the projection dynamic, with incentive $\incentive_i(x) = f_i(x) - \frac{1}{n}\sum_j{f_j(x)}.$ From \cite{Lahkar08} and \cite{Sandholm08}, we know that $||\ess{x} - x||^2$ is a Lyapunov function if $\ess{x}$ is ESS, so one would hope that ISS reduces to ESS in this case as well. The ISS condition $(\ess{x} / x) \cdot f(x) > 1 \cdot f(x),$ however, is not obviously the same as the ESS condition. If we assume a two-type linear fitness landscape
\[ f(x) = \left( \begin{matrix}
          a & b \\
          c & d 
          \end{matrix} \right) x, \]
it is a straightforward derivation to show that the ISS condition leads to
\[ \left[ (a-c) x_1 + (b-d) x_2 \right] \left(x_2 \ess{x}_1 - x_1 \ess{x}_2
\right) > 0,\]
which along with the constraint that $x_1 + x_2 = 1$ is the well-known condition for the existence of an internal ESS (and more generally, a Nash equilibrium). Similarly, we can give a family of examples that includes the projection dynamic. Consider the escort dynamics defined in \cite{Harper2011}, where $\escort$ is a positive non-decreasing function:
\begin{equation}
    \dot{x_i} = \escort_i(x) \left( f_i(x) - \frac{\sum_j{\escort_j(x) f_j(x)}}{\sum_j{\escort_j(x)}} \right).
\label{escort_dynamic}
\end{equation}
For these dynamics, an incentive is given by the right-hand-side of Equation \ref{escort_dynamic} and the ISS condition is \[ \sum_i{ \frac{\ess{x}_i \escort_i(x) f_i(x)}{x_i}} > \sum_i{ \frac{x_i \escort_i(x) f_i(x)}{x_i}}, \] which only clearly reduces to ESS for the replicator dynamic ($\escort(x) = x$). Nevertheless, it was shown in \cite{Harper2011} that an ESS, if it exists, is asymptotically stable for these dynamics, so it is desirable to have a generalization of Theorem \ref{thm_ess_iss} that captures this family as well.

To this end we introduce a functional parameter $\escort$ called an escort and related quantities.  An escort is a function $\escort$ that is nondecreasing and strictly positive on $(0, 1]$. A vector-valued function $\escort(x) = (\escort_1(x), \ldots, \escort_n(x))$ is called an escort if each component is an escort. We denote the normalized escort vector by $\escortdist$, i.e. normalized such that $\escortdist(x)$ is in the simplex. Generalized information divergences are defined by generalizing the natural logarithm using an escort function. See \cite{Harper2011} and \cite{Naudts04} for more details.

\begin{defn}[Escort Logarithm]
Define the escort logarithm
\[ \escortlog(x) = \int_{1}^{x}{\frac{1}{\escort{(v)}} \, dv} \]
\end{defn}

The escort logarithm shares several properties with the natural logarithm: it is negative and increasing on $(0,1]$ and concave on $(0,1]$ if $\escort$ is strictly increasing.

\begin{defn}[Escort Divergence]
Define the escort divergence
\[ D_{\escort}(x, y) = \sum_{i=1}^{n}{\int_{y_i}^{x_i}{ \escortlogi{(u)} -
\escortlogi{(y_i)} \, du} }.\]
\end{defn}

Since the logarithms are increasing on $(0,1)$, this divergence satisfies the usual properties of an information divergence on the simplex: $D(x, y) > 0$ if $x \neq y$ and $D(x, x) = 0$.

\section{The Time Scale Escort Incentive Dynamic}

Now we are able to define the time scale escort incentive dynamic and give the main theorems.
\begin{defn}
Define the time scale escort incentive dynamic to be
\[ \tsd{x_i} = \incentive_i(x) - \escortdist_i(x) \sum_{j}{\incentive_j(x)} \]
\label{tsed}
\end{defn}

We also need a definition of ISS that incorporates the escort parameter.

\begin{defn}
Define $\ess{x}$ to be an escort ISS (or EISS) for an incentive $\incentive$ and an escort $\escort$ if for all $x$ in some neighborhood of $\ess{x}$ the following inequality holds:
\[ \sum_{i}{ \frac{\ess{x}_i \incentive_i(x)}{\escort_i(x)}} > \sum_{i}{\frac{x_i \incentive_i(x_i)}{\escort_i(x)}} \]
\end{defn}
Remark: This definition can be understood in terms inner products of a Riemannian metric $g_{ij}(x) = \delta_{ij} / \escort_i(x)$.

For one example of a choice of escorts leading to interesting dynamics, consider the escort $\escort(x) = \beta x$. This introduces an intensity of selection parameter that alters the velocities of the dynamic (but not the stable point, if any). In fact, each type can have its own intensity of selection, and its own geometry. A popular choice of escort in information geometry is $\escort(x) = x^q$, which gives q-analogs of logarithms, exponentials, and divergences \cite{Naudts04}. See \cite{Harper2011} for more examples. The corresponding q-divergence is given by:

\[ D_q(x||y) = \begin{cases}
\frac{1}{2}||x - y||^2 & \text{if } q=0 \\
D_{KL}(x||y) & \text{if } q=1 \\
\sum_{i}{\left[\log{\frac{x_i}{y_i}} + \frac{y_i-x_i}{x_i} \right]} & \text{if } q=2 \\
\frac{1}{1-q}\sum_{i}{\left[ \frac{y_i^{2-q} - x_i^{2-q}}{2-q} - y_i^{1-q}(y_i-x_i)\right]} & \text{if } q \geq 0, q \neq 1,2 \\
\end{cases}\]

\begin{figure}[h]
        \begin{subfigure}[b]{0.5\textwidth}
                \centering
                \includegraphics[width=\textwidth]{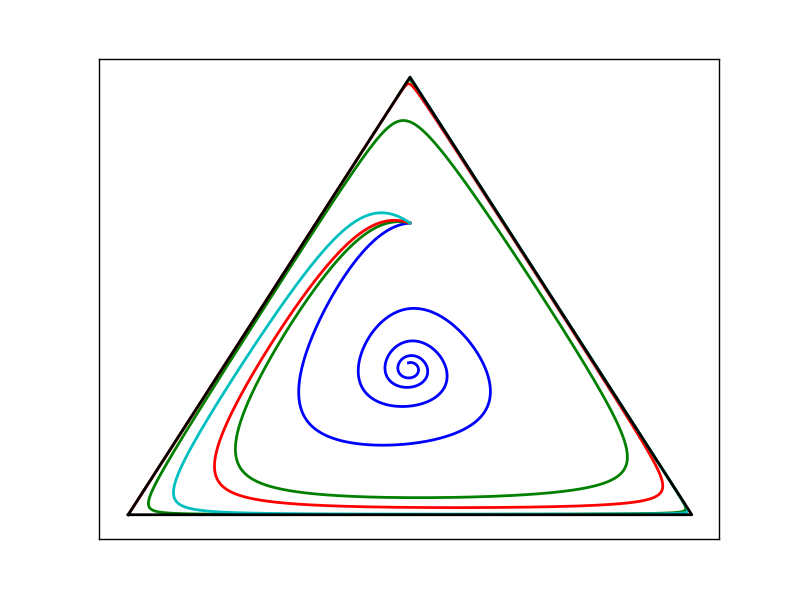}
%                 \caption{}
                \label{fig_4_1}
        \end{subfigure}%
        ~ %add desired spacing between images, e. g. ~, \quad, \qquad etc. 
          %(or a blank line to force the subfigure onto a new line)
        \begin{subfigure}[b]{0.5\textwidth}
                \centering
                \includegraphics[width=\textwidth]{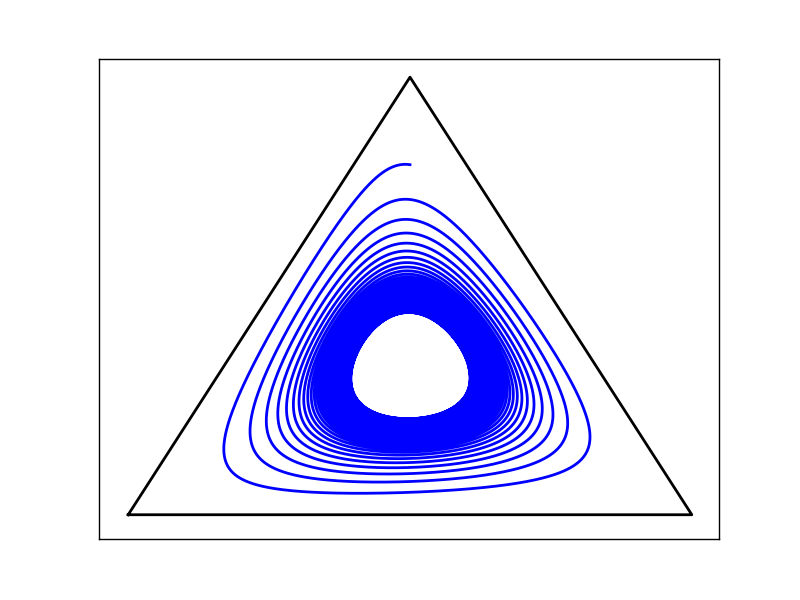}
%                 \caption{$p=0.5$}
                \label{fig_4_2}
        \end{subfigure}
        \caption{q-Escorts for various $q$, with replicator incentive given by a linear fitness landscape with RSP game matrix (a=1, b=2). Left: $q \in \{0.2, 0.8, 1.0, 2.0\}$ with colors blue, green, red, and cyan respectively; Right: $q=4$. The standard replicator equation diverges for $q=1$ but this is not always the case for the q-escort.}
        \label{fig_4}
\end{figure}

As in the case of the incentive dynamic, the continuous-time escort incentive dynamic is a special case of the escort replicator dynamic. Consider the escort exponential $\escortexp$, which is the functional inverse of the escort logarithm. It has the following crucial property:
\[ \frac{d}{dx}{\escortexp{x}} = \phi(\escortexp{x})\]

Now let $x_i = \escortexpi(v_i - G)$ and consider the derivative $\dot{x}_i = \escort(\escortexpi(v_i - G)) ( \dot{v}_i - \dot{G}) = \escort_i(x) (\dot{v}_i - \dot{G})$. Then if we can solve the following two auxillary equations we can also solve the escort incentive dynamic: \[ \dot{v}_i = \frac{\incentive_i(x)}{\escort_i(x)} \]
\[ \dot{G} = \frac{\sum_j{\incentive_j(x)}}{\sum_j{\escort_j(x)}} \]

This also shows how we can translate the escort incentive into the escort replicator equation. Given the escort incentive
\[ \dot{x}_i = \incentive_i(x) - \escortdist_i(x) \sum_{j}{\incentive_j(x)}, \]
define a fitness landscape \[f_i(x) = \frac{\incentive_i(x)}{\escort_i(x)}.\]
Then the escort incentive equation is an escort replicator equation, and the ESS condition $\hat{x} \cdot f(x) > x \cdot f(x)$ is the EISS condition
\[ \sum_{i}{ \frac{\ess{x}_i \incentive_i(x)}{\escort_i(x)}} > \sum_{i}{ \frac{x_i \incentive_i(x)}{\escort_i(x)}}. \]

That an $\ess{x}$ is ESS iff the escort divergence is a Lyapunov function was shown in \cite{Harper2011} for the continuous escort replicator dynamic. What remains to be shown now is the extension of this result to time scales $\mathbb{T} = h\mathbb{Z}$ for $0 < h \leq 1$, and the relationship between the concepts of escort ISS and ESS. Before doing so, consider the following examples. In Figure \ref{fig_4}, replicator incentives for RSP matrices are plotted with various $q$-incentives. Notice that the equilibrium is not an ESS for the fitness landscape $f$ yet the dynamic converges to the center of the simplex for some choices of the escort. In Figure \ref{fig_5} we have phase plots for dynamics with the same landscape but with both q-replicator incentives and q-escorts. In this case the conditions for E-ISS and ESS are the same, so all the dynamics converge. The KL-divergences are not monotonically decreasing in all cases, but the corresponding escort-divergences are in fact Lyapunov functions, as shown in Figure \ref{fig_6}. The dynamic in this case is given by:
\begin{equation}
\dot{x}_i = x_i^{q} f_i(x) - \frac{x_i^q}{\sum_{j}{x_j^{q}}} \sum_j{x_j^q f_j(x)}.
\label{double_q}
\end{equation}

\begin{figure}[h]
        \begin{subfigure}[b]{0.5\textwidth}
                \centering
                \includegraphics[width=\textwidth]{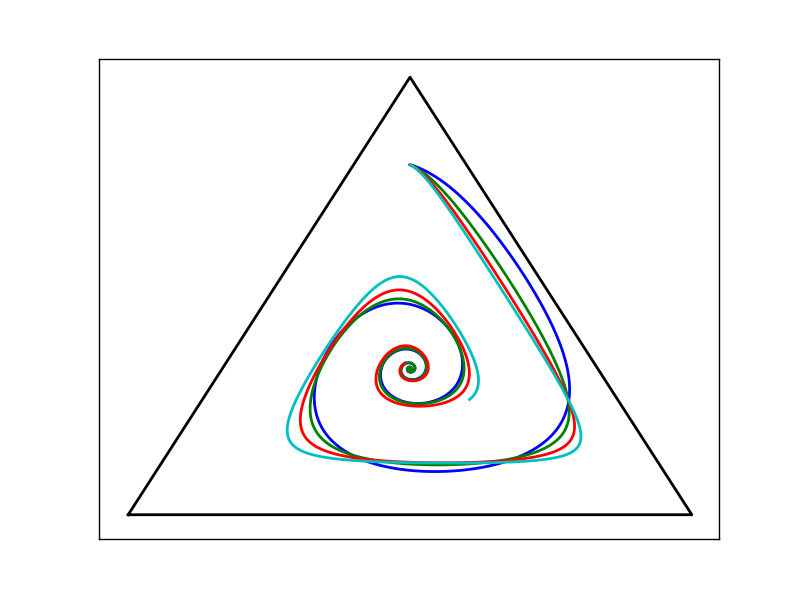}
%                 \caption{$p \in \{\}$ with colors blue, green, red, and cyan.....}
                \label{fig_5_1}
        \end{subfigure}%
        ~ %add desired spacing between images, e. g. ~, \quad, \qquad etc. 
          %(or a blank line to force the subfigure onto a new line)
        \begin{subfigure}[b]{0.5\textwidth}
                \centering
                \includegraphics[width=\textwidth]{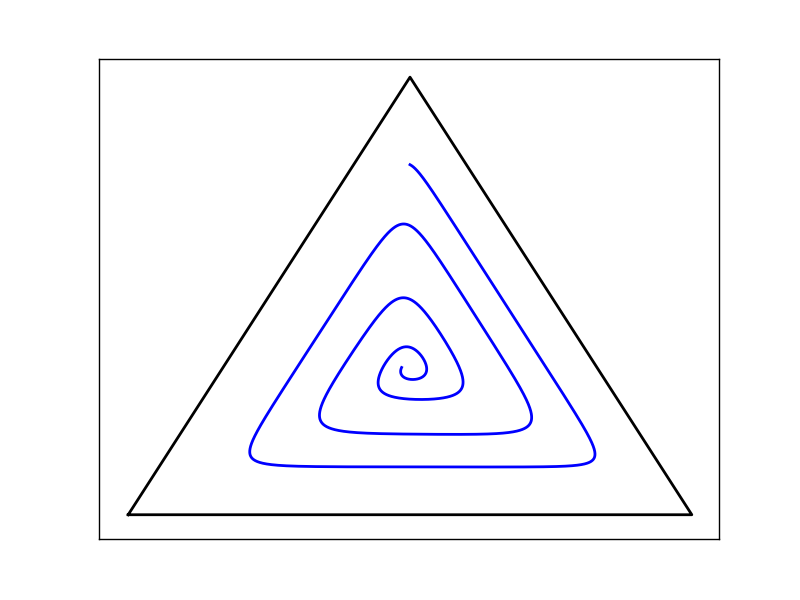}
%                 \caption{$p=4.0$}
                \label{fig_5_2}
        \end{subfigure}
        \caption{Plots of Equation \ref{double_q} for various $q$. Left: $0.5, 1.0, 1.5, 2.0$ blue, green, red, cyan, magenta, repsectively; Right: $q=4$.}
        \label{fig_5}
\end{figure}

\begin{figure}[h]
        \begin{subfigure}[b]{0.5\textwidth}
                \centering
                \includegraphics[width=\textwidth]{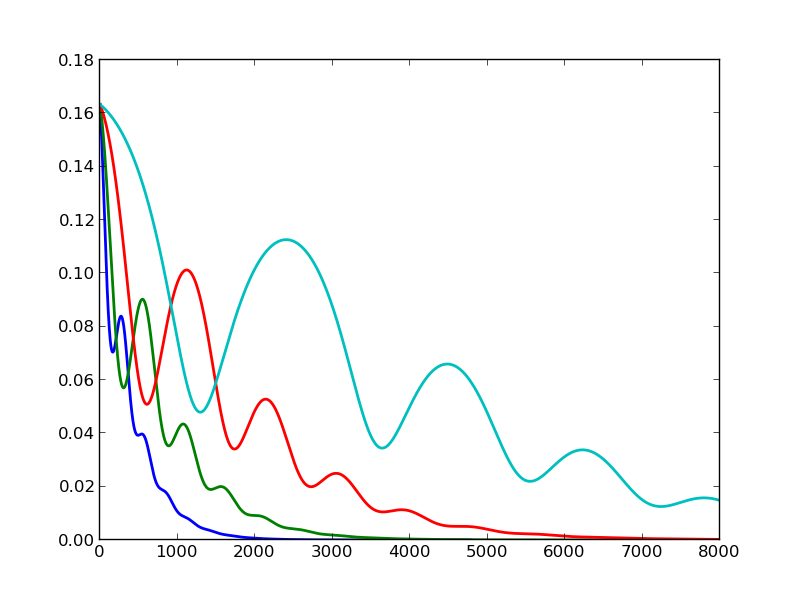}
%                 \caption{}
                \label{fig_6_1}
        \end{subfigure}%
        ~ %add desired spacing between images, e. g. ~, \quad, \qquad etc. 
          %(or a blank line to force the subfigure onto a new line)
        \begin{subfigure}[b]{0.5\textwidth}
                \centering
                \includegraphics[width=\textwidth]{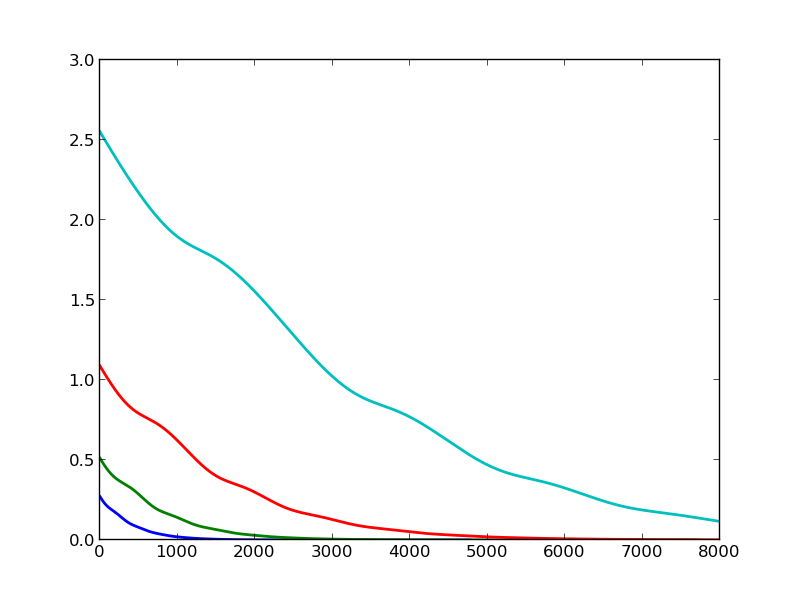}
%                 \caption{}
                \label{fig_6_2}
        \end{subfigure}
        \caption{Plots of candidate Lyapunov functions for q-escort and q-incentive. Left: KL-divergence for all dynamics; Right: q-divergences for the corresponding $q$ for each dynamic. Colors correspond to the dynamics in Figure \ref{fig_5} }
        \label{fig_6}
\end{figure} 

Now we state and prove the main theorem of this section.

% \begin{theorem}
% Consider an incentive of the form $\incentive_i(x) = \escort_i(x) g_i(x)$, where $g$ is such that $\sum_i {\escort_i(x) g_i(x)} = 0.$ If $g$ is payoff monotonic then an ESS is an EISS.\\
% Are we still claiming this?
% \end{theorem}
% \label{eiss_ess}

\begin{theorem}
Let $\incentive$ be an incentive function and let $\escort$ be an escort. If $\hat{x}$ is an escort ISS then $D(x) = \escortdiv(\ess{x}, x)$ is a local time scale Lyapunov function for the time scale escort incentive dynamics.
\label{thm_tsei}
\end{theorem}

The proof follows easily from the established facts and the following lemma.

\begin{lemma}
\[ \tsd{D}(x) \leq -\sum_{i}{\frac{\ess{x}_i - x_i}{\escort_i(x)}\tsd{x_i} } \]
Equality holds in the limit that $h \to 0$, i.e. for $\mathbb{T} = \mathbb{R}$.
\end{lemma}
% \begin{proof}
% Since escort functions are nondecreasing and positive on $(0,1]$, we have the
% following two facts:
% \[ \int_{a}^{b}{\frac{du}{\escort(u)}} \geq \frac{b-a}{\escort(a)}\]
% \[ \int_{a}^{b}{\escortlog(u) \, du} \geq (b-a) \escortlog(b) \]
% Now we have
% \begin{align*}
% h \tsd{D}(x) &= \sum_i{ \int_{x_i}^{x'_i}{ \escortlog(u) \, du}} -
% \sum_{i}{\left[ \left(\ecross{x'}{x'} - \ecross{\ess{x}}{x'} \right) -
% \left(\ecross{x}{x} - \ecross{\ess{x}}{x} \right) \right]} \\
% & \geq \sum_i{ \left(x'_i - x_i \right) \escortlog(x'_i)} -
% \sum_{i}{\ecross{x'}{x'} - \ecross{\ess{x}}{x'} - \ecross{x}{x} +
% \ecross{\ess{x}}{x}} \\
% &= \sum_i{ \left(\ess{x}_i - x_i \right) \left( \escortlogi{x'_i} -
% \escortlogi{x_i} \right) } \\
% &= \sum_i{ \left(\ess{x}_i - x_i \right) \int_{x_i}^{x'_i}{
% \frac{du}{\escort_i(u)} }}\\
% &\geq \sum_i{ \left(\ess{x}_i - x_i \right) \frac{x'_i - x_i}{\escort_i(x_i)}}\\
% \end{align*}
% Bringing $h$ to the right hand side completes the lemma for $h > 0$. Equality
% for $h=0$ (i.e. $\mathbb{T} = \mathbb{R}$) can be directly verified with
% differentiation.
% \end{proof}

\begin{proof}
Since escort functions are nondecreasing and positive on $(0,1]$, we have the following two facts:
\[ \int_{a}^{b}{\frac{du}{\escort(u)}} \leq \frac{b-a}{\escort(a)}\]
\[ \int_{a}^{b}{\escortlog(u) \, du} \leq (b-a) \escortlog(b) \]
Now we have
\begin{align*}
h \tsd{D}(x) &= \sum_i{ \int_{x_i}^{x'_i}{ \escortlog(u) \, du}} -
\sum_{i}{\left[ \left(\ecross{x'}{x'} - \ecross{\ess{x}}{x'} \right) -
\left(\ecross{x}{x} - \ecross{\ess{x}}{x} \right) \right]} \\
&\leq \sum_i{ \left(x'_i - x_i \right) \escortlog(x'_i)} -
\sum_{i}{\ecross{x'}{x'} - \ecross{\ess{x}}{x'} - \ecross{x}{x} +
\ecross{\ess{x}}{x}} \\
&= \sum_i{ \left(\ess{x}_i - x_i \right) \left( \escortlogi{x'_i} -
\escortlogi{x_i} \right) } \\
&= -\sum_i{ \left(\ess{x}_i - x_i \right) \int_{x_i}^{x'_i}{
\frac{du}{\escort_i(u)} }}\\
&\leq -\sum_i{ \left(\ess{x}_i - x_i \right) \frac{x'_i -
x_i}{\escort_i(x)}}\\
\end{align*}
Bringing $h$ to the right hand side completes the lemma for $h > 0$. Equality for $h=0$ (i.e. $\mathbb{T} = \mathbb{R}$) can be directly verified with differentiation (and is given in \cite{Harper2011}).
\end{proof}

To complete the proof of Theorem \ref{thm_tsei} we need only apply the lemma to the respective dynamics, use the definition of EISS, and verify the two technical conditions of the time scale Lyapunov theorem.
\begin{proof}[Proof of theorem \ref{thm_tsei}]
Using the lemma and substituting the right-hand side of equation \ref{tsed} gives:
\begin{align*}
\tsd{D}(x) &\leq -\sum_{i}{\frac{\ess{x}_i - x_i}{\escort_i(x)}\tsd{x_i} }\\
% &= (x_i - \ess{x}_i) \left( \dot{v}_i - \dot{G}  \right) \\
&= -\sum_{i}{\frac{\ess{x}_i - x_i}{\escort_i(x)}}\left(\incentive_i(x) - \escortdist_i(x) \sum_{j}{\incentive_j(x)}\right)\\
&= (x_i - \ess{x}_i) \left( f_i(x) - \bar{f}(x)  \right) = x \cdot f(x) - \ess{x} \cdot f(x) 
\end{align*}
where $f$ is the effective landscape. If $\ess{x}$ is an EISS, then the right-hand side is negative.
\end{proof}

Throughout this paper, it may appear that all the examples we have given are for continuous dynamics. In fact, every phase plot in this paper is for a discrete time-scale with $h \approx 1/100 \to 1/1000$ unless otherwise indicated, and in all previously known cases, the results have coincided with the expectation for the continuous dynamics. Hence all the examples in this paper illustrate the main theorem for these particular time-scales.

Let us return to the example of the projection dynamic. Previously we saw that the orthogonal projection dynamic was obtainable from the incentive $\incentive_i(x) = f_i(x) - (1/n) \sum_j{f_j(x)}$ and the escort $\escort(x) = x$ on interior trajectories, but the ISS condition did not appear to be the same as ESS (though we were able to argue for two-player linear landscapes that they are equivalent). With the generalized theorem, we can also obtain the dynamic from the incentive $\incentive_i(x) = f_i(x)$ and $\escort(x) = 1$, and now the EISS condition reduces immediately to ESS. Moreover, the Lyapunov function given by the escort information divergence is, remarkably, one-half the Euclidean distance: $D(x) = (1/2) ||\ess{x} - x||^2$, capturing the known result up to a constant factor of one-half \cite{Nagurney97} \cite{Lahkar08}.

From the these results it should be clear that if the incentive factors as $\incentive_i(x) = \escort_i(x) g_i(x)$ then the condition for EISS will be the condition for ESS for the landscape $g$. So in particular if the function $g$ is payoff monotonic, these dynamics have similar interior dynamics (though with different Lyapunov functions). If the incentive factorizes as $\incentive_i(x) = x_i \escort_i(x_i) g_i(x)$, then the ESS criterion is the ISS criterion for the incentive $g$. So while the theorem covers all combinations of valid incentives and escorts, it is clear that careful choices may lead to simplifications in the stability criterion. The projection dynamic is somewhat special in that it can be described equivalently by multiple choices of escort and incentive. It is also clear that while there exist known Lyapunov functions for some dynamics that are not the KL-divergence (or a generalization) such as the best reply dynamic, the incentive dynamics approach yields a commonly-derived and motivated Lyapunov function.

\subsection{A Variation of the Replicator equation}
We can also define novel dynamics by combing escorts and incentives. Consider the following dynamic, using the replicator incentive $\incentive_i(x) = x_i f_i(x)$ and the escort $\escort_i(x) = 1$: Equation (\ref{tsed}) gives the dynamic as
\[ \dot{x}_i = x_i f_i(x) - \frac{1}{n}\bar{f}(x).\]
In this case the EISS criterion is $\sum_i{\ess{x}_i x_i f_i(x)} > \sum_i{x_i x_i f_i(x)}$, which is the same for the dynamic obtained with $\incentive_i(x) = x_i^2 f_i(x)$ and $\escort_i(x_i) = x_i$. These two dynamics differ qualitatively, however, as we can see for a constant fitness landscape $f_i(x) = 1$ for all $i$. While both dynamics have the barycenter of the simplex of rest points, the former is not forward-invariant on the interior and has no other rest points, while the latter has rest points on the boundary and the barycenters of the boundary simplices. Moreover for an RSP matrix, while the replicator equation converges to an interior ESS, the projection variant does not. In general, this dynamic may not be forward invariant on the simplex since while the sum $\sum_i{\dot{x_i}} = 0$, the individual components can become negative for certain states and fitness landscapes. 

% \begin{figure}[h]
%     \centering
%     \includegraphics[width=0.5\textwidth]{figure_8.png}
%     \caption{}
%     \label{fig_8}
% \end{figure} 

\subsection{Best Reply Variants}
Similarly, we can define a projection best-reply dynamic using p-escorts. See Figure \ref{fig_7}. Note that for $p=0.5$, the dynamic is not forward-invariant on the interior. Starting near the barycenter yields ``Shapley-like-triangles''. These can be extended cycles that jump on and off the boundary if the best reply incentive is combined with the projection incentive (cyan curves in Figure \ref{fig_7}). Using just the projection escort (i.e. $q=0$) would not yield this behavior (trajectories remain on the boundary once they reach it).

\begin{figure}[h]

\end{figure} 

\begin{figure}[h]
        \begin{subfigure}[b]{0.5\textwidth}
            \centering
            \includegraphics[width=\textwidth]{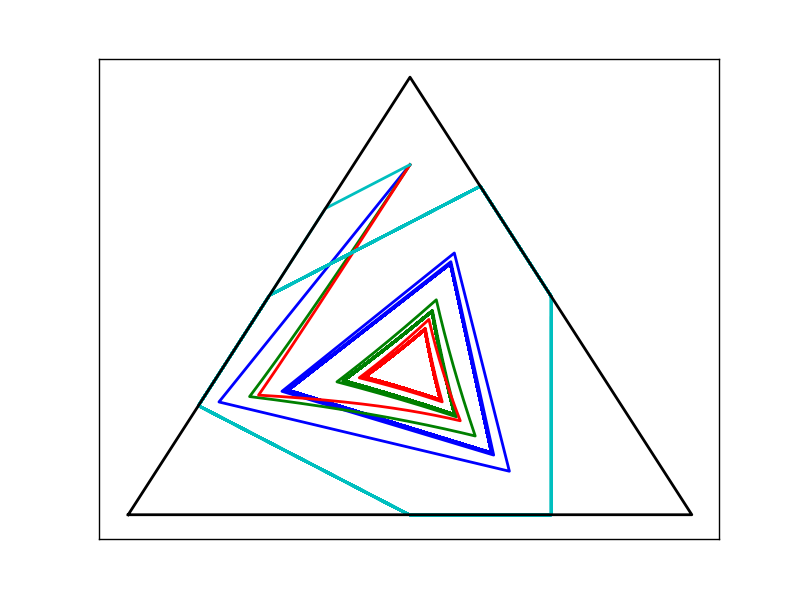}
%             \caption{.}
            \label{fig_7_1}
        \end{subfigure}%
        ~ %add desired spacing between images, e. g. ~, \quad, \qquad etc. 
          %(or a blank line to force the subfigure onto a new line)
        \begin{subfigure}[b]{0.5\textwidth}
            \centering
            \includegraphics[width=\textwidth]{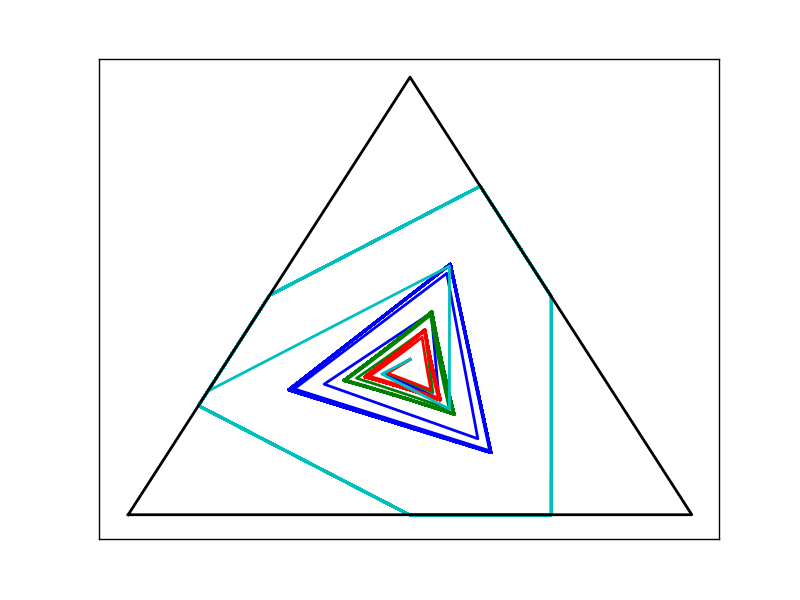}
%             \caption{Best Reply incentives with $q=0, 1, 2, 3$.}
            \label{fig_7_2}
        \end{subfigure}
        \caption{Left: Best reply incentives with $q=0.6, 1, 3$ and combined with projection incentive at initial point $(1/10, 1/10, 8/10)$ blue, green, red, respectively; Right: Same with initial point $(1/3.1, 1/3.1, 1.1/3.1)$.}
        \label{fig_7}
\end{figure} 

\subsection{Incentive-Escort Dynamic is just the Incentive Dynamic; Zero-Sum Incentives}

We have seen that the incentive-escort dynamic can be rewritten as the escort dynamic; it can also be rewritten as the incentive dynamic by simply taking the right-hand side as the incentive. As we have seen, however, one may miss the stability of an internal equilibrium by trying to apply the KL-divergence as a Lyapunov function since it is not necessarily monotonic. So while the decomposition of an incentive dynamic into an incentive-escort dynamic does not yield a larger class of dynamics, it does yield stability theorems for more members of the class.

Zero-sum incentives, when $\sum_{i}{\incentive_i(x)} = 0$ are informative here. First, notice that in this case the escort-incentive dynamic reduces to just the incentive dynamic, producing the exact same trajectory for any escort. One such example is the replicator incentive for a RSP matrix with $a=1=b$. On the other hand, any incentive can be made into a zero sum incentive by subtracting $y \left(\sum_{i}{\incentive_i(x)}\right)$ from the incentive for any value of $y$ in the simplex, and the choice of $y$ can vary at each point in the simplex. In other words, it requires a choice of escort, or more generally, a choice of Riemannian metric. In this case, there is a preferred choice of incentive, and so a natural candidate for a Lyapunov function.

\section{Time-Scale Lyapunov functions for Adaptive Dynamics}

As suggested in the last section, we now consider the adaptive dynamics for a Riemannian metric $G$ on the simplex, defined in \cite{Hofbauer90}. Let $g = G^{-1} \mathbf{1}$ and $C = G^{-1} - (g^T \mathbf)^{-1} g g^T$. Then the time-scale adaptive dynamics for the metric G and a fitness landscape $f$ is $x^{\Delta}_i = \sum_{j}{C_{ij} f_j}$. The adaptive dynamics includes the escort dynamics as a special case. A geodesic approximation is shown to be a local Lyapunov function in \cite{Hofbauer90} for the continuous dynamic; to obtain a time-scale Lyapunov function we use a global divergence, which we now give define. Let
\[ D_G(x) = D(\hat{x} || x) = \sum_{i,j}{\int_{x_i}^{\hat{x}_i}{ \left(\log G\right)_{ij}(v) - \left(\log G \right)_{ij}(\hat{x}_i) \, dv}},\]
where \[ \left(\log G\right)_{ij}(x) = \int_{1}^{x}{G_{ij}(v) \, dv}.\]

\begin{theorem}
Let $G$ be a Riemannian metric and assume that the reciprocal of each component $G_{ij}$ is an escort function. Then
\begin{enumerate}
    \item $D_G$ is an information divergence; that is $D_G(\hat{x}) = 0$ and $D_G(x) > 0$ for $x \neq \hat{x}$
    \item $D$ is a local time scale Lyapunov function for the adaptive dynamics if $\hat{x}$ is an ESS.
\end{enumerate}
\label{thm_g_ess}
\end{theorem}
The proof is almost identical to the proof of theorem \ref{thm_tsei}, and so is omitted. Note that $D_G(x) \leq (x - 
\hat{x})^{T} G(x) (x - \hat{x})$. One can now play the same game as before, identifying the incentive in terms of the fitness landscape, and forming best reply, logit, projection, or any other incentive dynamics with respect to particular Riemannian geometries. Indeed, we have that $\incentive_i(x) = \sum_j{G^{-1}_{ij} f_j(x)}$ and the adaptive dynamics can be written as
\begin{equation}
x_i^{\Delta} = \incentive_i(x) - \frac{(G^{-1} \mathbf{1})_i}{\sum_j{(G^{-1} \mathbf{1})_j}} \sum_j{\incentive_j(x)}. 
\label{metric_incentive}
\end{equation}

Call this the \emph{metric-incentive dynamic}. As before, we can identify the idea of an ESS with that of a G-ISS, i.e. a state $\hat{x}$ is a G-ISS if for all x in a neighborhood of $\hat{x}$, the inequality $\incentive(x) \cdot G(x) (\ess{x} - x) > 0$. Then we these definitions, we can restate Theorem \ref{thm_g_ess} (2) as follows.
\begin{theorem}
Let $G$ be a Riemannian metric and assume that the reciprocal of each component $G_{ij}$ is an escort function. Then \item $D_G$ is a local time scale Lyapunov function for the metric-incentive dynamics if $\hat{x}$ is a G-ISS.
\end{theorem}

\section{Multiple Populations}

Following equation \ref{metric_incentive}, we can formulate a multiple population dynamic in which each population operates on its own incentive, time-scale, and geometry. Let $\hat{G}(x) = \frac{(G^{-1} \mathbf{1})}{\sum_j{(G^{-1} \mathbf{1})_j}}$ be the vector of coefficients in equation \ref{metric_incentive}. Then the multiple population time-scale metric incentive dynamic is (populations indexed by $\alpha$):

\begin{equation}
x^{\Delta_{\alpha}} = \incentive_{i, \alpha}(x_{\alpha}) - \hat{G}_{i, \alpha}(x_{\alpha}) \sum_j{\incentive_{j, \alpha}(x)}
\label{multipop_metric_incentive}
\end{equation}

\subsection{Examples}

We give two examples before discussing stability in Figures \ref{fig_8} and \ref{fig_9}. The only difference between the two examples is the incentive for the second population: in the first case, the incentive is logit, and in the second, q-replicator with $q=2$. Nevertheless, the resulting dynamics are very different.

\begin{figure}[h]
        \begin{subfigure}[b]{0.6\textwidth}
            \centering
            \includegraphics[width=\textwidth]{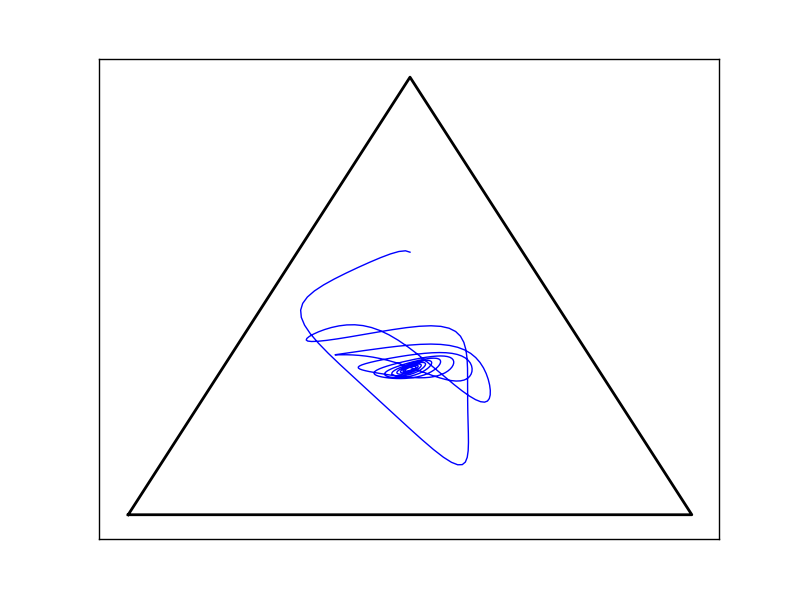}
%             \caption{}
            \label{fig_8_1}
        \end{subfigure}%
        ~ %add desired spacing between images, e. g. ~, \quad, \qquad etc. 
          %(or a blank line to force the subfigure onto a new line)
        \begin{subfigure}[b]{0.6\textwidth}
            \centering
            \includegraphics[width=\textwidth]{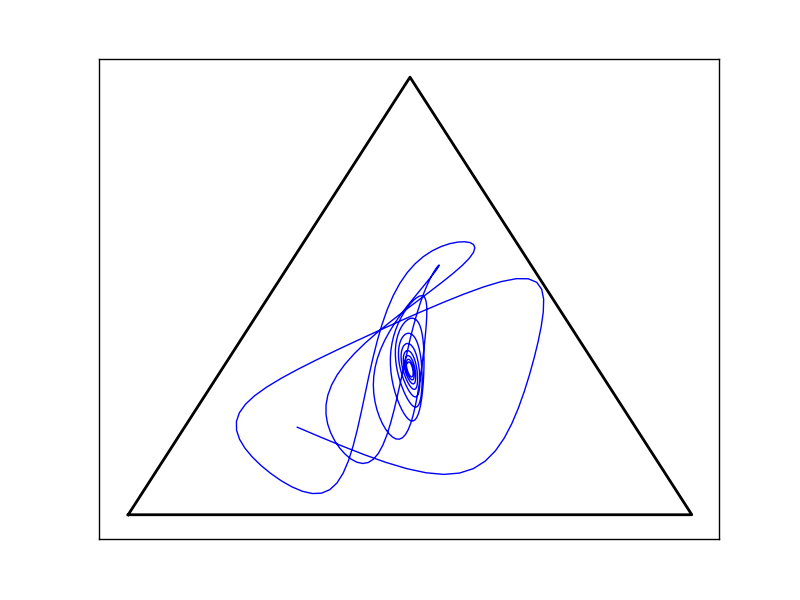}
%             \caption{}
            \label{fig_8_2}
        \end{subfigure}
        \caption{Two population dynamic with $h=1/10$ for both populations, with fitness landscape given by $a=-1$, $b=-2$. Both populations converge to the center (500 iterations shown). Left: Replicator incentive, Shahshahani geometry, initial point $(1/5,1/5,3/5)$; Right: Logit incentive, $\eta=0.4$, Euclidean geometry, initial point $(3/5,1/5,1/5)$}
        \label{fig_8}
\end{figure} 

\begin{figure}[h]
        \begin{subfigure}[b]{0.6\textwidth}
            \centering
            \includegraphics[width=\textwidth]{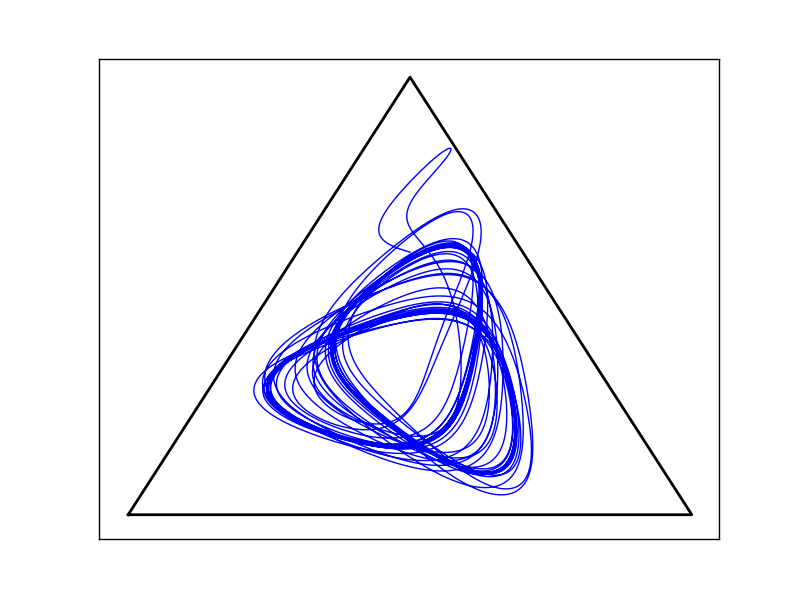}
%             \caption{}
            \label{fig_9_1}
        \end{subfigure}%
        ~ %add desired spacing between images, e. g. ~, \quad, \qquad etc. 
          %(or a blank line to force the subfigure onto a new line)
        \begin{subfigure}[b]{0.6\textwidth}
            \centering
            \includegraphics[width=\textwidth]{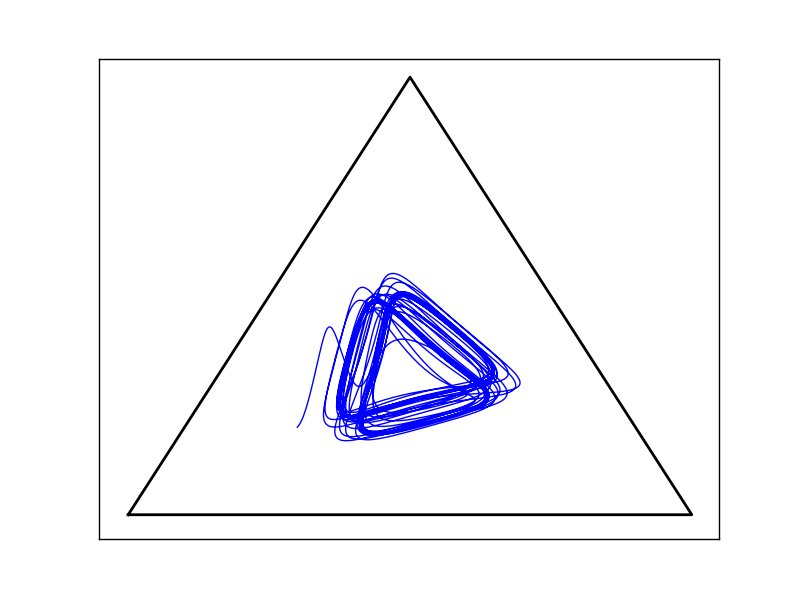}
%             \caption{}
            \label{fig_9_2}
        \end{subfigure}
        \caption{Two population dynamic with $h=1/10$ for both populations, with fitness landscape given by $a=-1$, $b=-2$. Neither population converges after 10,000 iterations. Left: Replicator incentive, Shahshahani geometry, initial point $(1/5,1/5,3/5)$; Right: q-Replicator incentive with $q=2$, Euclidean geometry, initial point $(3/5,1/5,1/5)$}
        \label{fig_9}
\end{figure}

In the spirit of \cite{Fryer2012} and \cite{cressman1997local}, if there is a multiple-population G-ISS where the time-scales do not differ for each population then we can find a time-scale Lyapunov function for the system by summing $D_G$ for each population. The proof is again analogous to \ref{thm_tsei}, and can easily be verified with differentiation for the continuous case.

\begin{theorem}
Suppose each population in system $\ref{multipop_metric_incentive}$ is of the same time scale ($\mathbb{T} = h \mathbb{Z}$ or $\mathbb{T} = \mathbb{R}$). Let $L = \sum{D_{G_{\alpha}}(\hat{x}_{\alpha} || x_{\alpha})}$. $L$ is a local time-scale Lyapunov function for the system \ref{multipop_metric_incentive} iff
\[ \sum_{\alpha}{\incentive_{\alpha}(x_{\alpha}) \cdot G_{\alpha}(x_{\alpha}) (\hat{x}_{\alpha} - x_{\alpha})} > 0 \]
for some neighborhood of $\hat{x}_{\alpha}$.
\end{theorem}

\subsection{Variable Time-Scales}

A suitable generalization of the time-scale Lyapunov theorem that would apply to a system on multiple time-scales does not seem to exist in the literature. Rather than prove such a result in this paper, we will just give an example relevant to the time-scales under discussion, and conjecture further results.

First consider two populations, on time scales $\mathbb{T}_1 = h \mathbb{Z}$ and $\mathbb{T}_2 = h / 2 \mathbb{Z}$ respectively. We must decide how to compute the derivative of the (candidate) Lyapunov function. In this simple case, simply taking the intersection of the time scales (resulting in just the the time scale $\mathbb{T}_1$ and computing the derivative at these points produces a positive definite and decreasing quantity, albeit on a subsequence of the original dynamical system. See Figure \ref{fig_10} for an example. The sum of the appropriate divergences is monotonically decreasing to zero as the populations converge. In this example, if $h=1/10$ for both populations, the Lyapunov quantity is not monotonic (it has one local maximum) globally, just locally.

\begin{figure}[h]
        \begin{subfigure}[b]{0.5\textwidth}
            \centering
            \includegraphics[width=\textwidth]{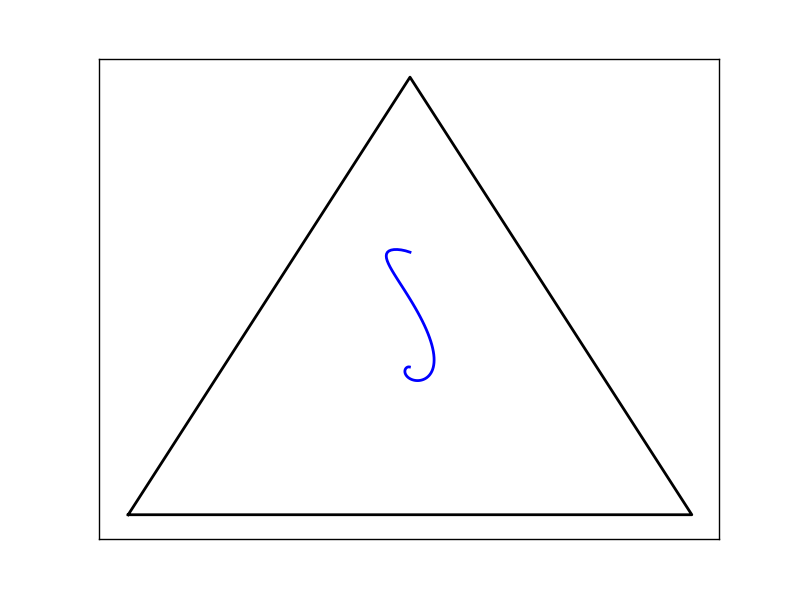}
            \caption{}
            \label{fig_10_1}
        \end{subfigure}%
        ~ %add desired spacing between images, e. g. ~, \quad, \qquad etc. 
          %(or a blank line to force the subfigure onto a new line)
        \begin{subfigure}[b]{0.5\textwidth}
            \centering
            \includegraphics[width=\textwidth]{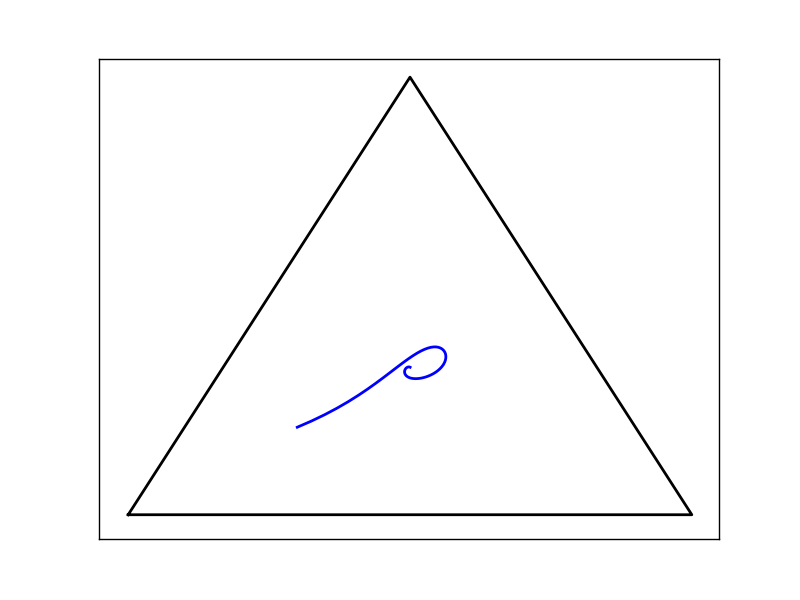}
            \caption{}
            \label{fig_10_2}
        \end{subfigure}\\
        \begin{subfigure}[b]{0.5\textwidth}
            \centering
            \includegraphics[width=\textwidth]{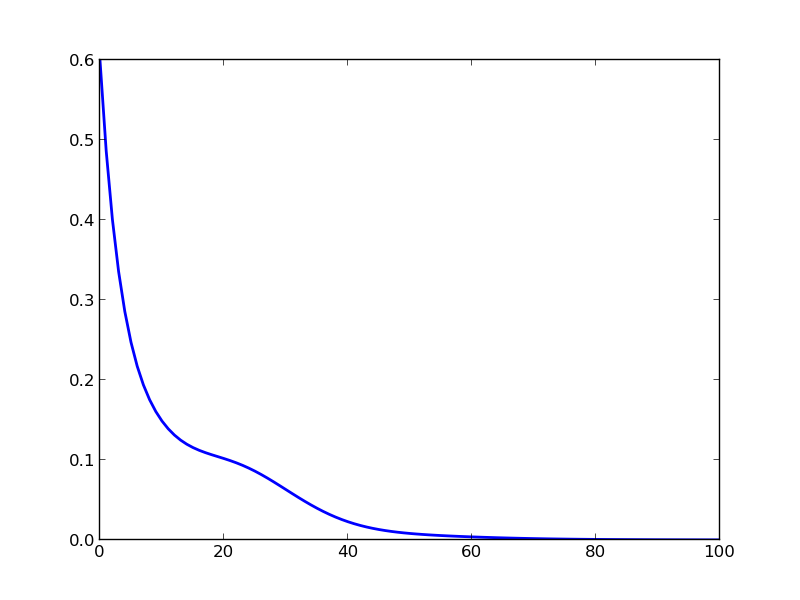}
            \caption{}
            \label{fig_10_3}
        \end{subfigure}
        \caption{Two population dynamic with $h=1/10$ for first population, $h=1/20$ for the second. Left (a): $a=-1$, $b=-2$, replicator incentive, Shahshahani geometry, initial point $(1/5,1/5,3/5)$; Right (b): $a=-1$, $b=-4$, replicator incentive, $q=2$ escort geometry, initial point $(3/5,1/5,1/5)$; Bottom (c): KL-divergence from center of first population plus $q=2$ divergence from the center of second population.}
        \label{fig_10}
\end{figure} 

In more general cases, depending on the time-scales involved, it may not be so obvious to determine what the minimal or intersection time-scale should be (since it could be the case that the intersection is empty, especially if there is an interaction between time scales of the form $\mathbb{T} = h \mathbb{Z}$ and say quantum time-scales of the form $\mathbb{T} = q^{\mathbb{Z}}$). Similarly, it may be the case that some type of dot product of time-scale derivatives with the respective Lyapunov functions for each population is the appropriate operation, but this is beyond the scope of this paper, and will undoubtedly be discussed in future work.

\section{Mutation}

The replicator-mutator equation is an evolutionary dynamic that takes the following form \cite{hofbauer1985selection} \cite{Page02}:
\[\dot{x}_i = \sum_{j}{ x_j \mu_{ji} f_j(x)} - x_i \bar{f}(x),\]
where $\mu_{ji}$ is the mutation transition probability of type $j$ mutating into type $i$ at reproduction. Using the relationship $\incentive_i(x) = x_i f_i$, we can translate this equation to incorporate incentives:
\[ \dot{x}_i = \sum_{j}{ \incentive_j(x) \mu_{ji}} - x_i \sum_{j}{\incentive_j(x)}\]

In vector form, we have that 
\[ \dot{x} = \incentive(x)^T \mu - x \sum_{j}{\incentive_j(x)} = \incentive(x)^T \mu - x (\incentive(x)^T \mathbf{1}) \]

Finally, we can incorporate time scales and geometric structure into a dynamic we call the \emph{time-scale metric incentive-mutator dynamic}:
\[ \tsd{x} = \incentive(x)^T \mu - \hat{G}(x) (\incentive(x)^T \mathbf{1})\]

A common form of the discrete replicator mutator equation \cite{Page02} can be obtained using $h=1$ and the incentive 
\[ \incentive_j(x) = \frac{x_j(f_j(x) - \bar{f}(x))}{\bar{f}(x)} \]

This gives:
\[ x_i' = \sum_{j}{\frac{x_j f_j(x) \mu_{ji}}{\bar{f}(x)}} \]

A common choice for mutation matrix is $\mu_{\epsilon} = (1-\epsilon) I_n - \epsilon / (n-1) (1_n - I_n)$, i.e. from the identity matrix, subtract $\epsilon$ from the diagonal and split uniformally over the other $n-1$ elements of the row. We end with a final example of a best reply dynamic with mutation in Figure \ref{fig_11}. See \cite{sandholm2011population} for another approach to a best reply dynamic with mutation.

\begin{figure}[h]
        \centering
        \includegraphics[scale=0.5]{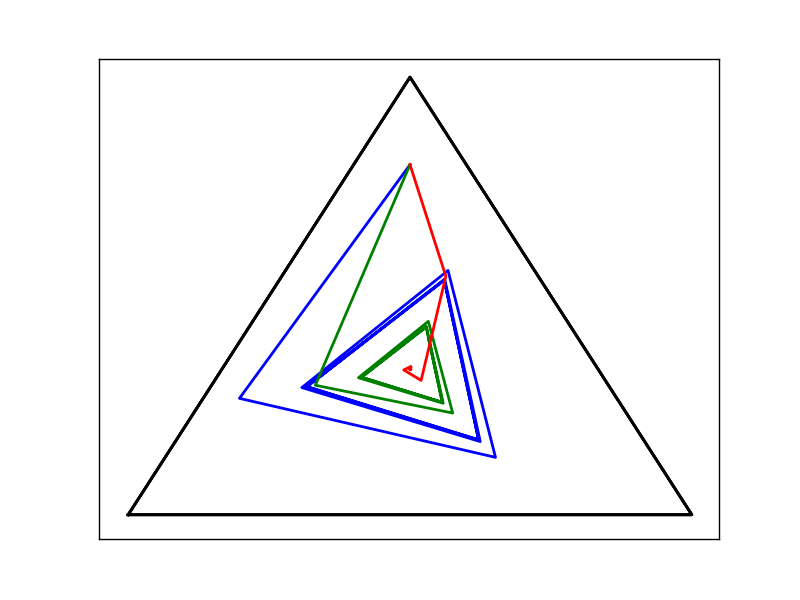}
        \caption{Best reply dynamic with uniform mutation matrix, starting at $(1/10, 1/10, 8/10)$ with fitness given by the RSP matrix with $a=1$ and $b=2$. Blue: $\epsilon=0.1$, Green: $\epsilon=0.4$, Red: $\epsilon=0.8$. The addition of mutation changes the size of the Shapley triangle depending on the value of $\epsilon$.}
        \label{fig_11}
\end{figure} 

\section{Discussion}
In this paper we have shown that a vast array of evolutionary dynamics can be analyzed in terms of incentives, Riemannian metrics, and divergence functions, the last of which allows the stability analysis of many discrete dynamics associated to the same geometric and game-theoretic constructions. In particular, we have shown that choice of incentive and Riemannian metric can affect the presence of equilibria, and that decomposition of dynamics into incentives, escorts, and Riemannian metrics can lead directly to Lyapunov functions for particular dynamics. In the process, we introduced a new information divergence defined in terms of a Riemannian metric rather than the typical dynamics formed only from metrics with diagonal matrix representations.

In our approach to evolutionary dynamics we focused attention on a few special cases that yield a particularly nice set of examples. It is possible to define these dynamics on arbitrary time scales and achieve some analogous results. It may be possible to formulate multipopulation dynamics where each population evolves according to a different time scale, which would require a substantial expansion of the current state of stability theorems on time scales. We have also shown that the time scale calculus and its stability theory can be a useful and unifying addition to the study of evolutionary dynamics.

\section{Acknowledgements}
All plots in this paper we created with \emph{matplotlib} and python code available at \url{https://github.com/marcharper/python-ternary} and \url{https://github.com/marcharper/metric-incentive-dynamics}. Marc Harper acknowledges support from the Office of Science (BER), U.S. Department of Energy, Cooperative Agreement No. DE-FC02-02ER63421.

\bibliography{ref}
\bibliographystyle{plain}

\end{document}